\documentclass[a4paper,11pt,titlepage,leqno,twoside]{article} 

\usepackage[pdftex]{hyperref}
\hypersetup{
    colorlinks=true, 
    linktoc=all,     
    linkcolor=blue,  
    allcolors=blue,
}

\usepackage[top=1.5in,bottom=2in,left=1.3in,right=1.3in,marginparwidth=1in]{geometry} 

\usepackage{mathrsfs}
\usepackage{amssymb}
\usepackage{amsmath}
\usepackage{amsfonts}
\usepackage{longtable}
\usepackage{paralist}
\usepackage[mathlines]{lineno}

\newtheorem{theorem}{Theorem}[section]
\newtheorem{Claim}[theorem]{Claim}
\newtheorem{definition}[theorem]{Definition}
\newtheorem{Cor}[theorem]{Corollary}
\newtheorem{fact}[theorem]{Fact}
\newtheorem{lemma}[theorem]{Lemma}

\newtheorem{remark}[theorem]{Remark}

\newtheorem{que}[theorem]{Question}

\newenvironment{proof}
{\noindent \textsc{Proof.}}
{\hspace*{\fill}$\Box$\bigskip}

\newcommand{\dom}[1]{\ensuremath{\mathrm{dom}}(#1)}

\newcommand{\set}[2]{\ensuremath{\{#1 \,|\, #2 \}}}
\newcommand{\seq}[2]{\ensuremath{\langle #1 \,|\, #2 \rangle}}

\newcommand{\restr}[2]{\ensuremath{#1 \! \upharpoonright \! #2}}

\newcommand{\iso}{\cong}
\newcommand{\sub}{\subseteq}

\newcommand{\then}{\rightarrow}

\newcommand{\bb}{\mathbb}

\newcommand{\beq}{\begin{equation}}
\newcommand{\eeq}{\end{equation}}
\newcommand{\brm}{\begin{remark}\begin{rm}}
\newcommand{\erm}{\end{rm}\end{remark}}
\newcommand{\mx}{\mathrm}
\newcommand{\bce}{\begin{compactenum}}
\newcommand{\ece}{\end{compactenum}}

\newcommand{\Add}{\mathrm{Add}}

\newcommand{\R}{\bb{R}}
\newcommand{\Q}{\bb{Q}}
\renewcommand{\P}{\bb{P}}
\newcommand{\TP}{{\sf TP}}

\newcommand{\T}{\bb{T}}
\newcommand{\M}{\bb{M}}
\newcommand{\x}{\times}

\newcommand{\E}{\bb{E}}
\newcommand{\Z}{\bb{Z}}
\renewcommand{\S}{\bb{S}}

\newcommand{\ZFC}{\sf ZFC}

\newcommand{\GCH}{\sf GCH}

\newcommand{\no}{\noindent}

\newcommand{\D}{\bb{D}}

\newcommand{\rest}[0]{\!\restriction\!}

\makeatletter
\renewcommand\section{\@startsection {section}{1}{\z@}%
                                   {-3.5ex \@plus -1ex \@minus -.2ex}%
                                   {2.3ex \@plus.2ex}%
                                   {\noindent\center\textsc}}
\renewcommand\subsection{\@startsection{subsection}{2}{\z@}%
                                     {-3.25ex\@plus -1ex \@minus -.2ex}%
                                     {1.5ex \@plus .2ex}%
                                     {\noindent\center\textsc}}
\renewcommand\subsubsection{\@startsection{subsubsection}{3}{\z@}%
                                     {-3.25ex\@plus -1ex \@minus -.2ex}%
                                     {1.5ex \@plus .2ex}%
                                     {\noindent\center \textsc}}
\makeatother

\usepackage{fancyhdr}
\begin{document}\thispagestyle{empty}

\begin{center}
\no {\large \MakeUppercase{Easton's theorem for the tree property below} $\aleph_\omega$ \par}

\bigskip

\medskip

{\v S}{\'a}rka Stejskalov{\'a}

\medskip

\begin{footnotesize}
Charles University, Department of Logic,\\
Celetn{\' a} 20, Praha 1, 
116 42, Czech Republic\\
sarka.stejskalova@ff.cuni.cz\\
\end{footnotesize}
\no \footnotesize{\today}
\end{center}

\bigskip

\begin{quote}
{\bf Abstract.}  Starting with infinitely many supercompact cardinals, we show that the tree property at every cardinal $\aleph_n$, $1 < n <\omega$, is consistent with an arbitrary continuum function below $\aleph_\omega$ which satisfies $2^{\aleph_n} > \aleph_{n+1}$, $n<\omega$. Thus the tree property has no provable effect on the continuum function below $\aleph_\omega$ except for the restriction that the tree property at $\kappa^{++}$ implies $2^\kappa>\kappa^+$ for every infinite $\kappa$.
\end{quote}

\emph{MSC}: 03E35; 03E55

\emph{Keywords}: Easton's theorem; Tree property; Large cardinals

\tableofcontents

\section{Introduction}

Recall that the continuum function is the function which maps an infinite cardinal $\kappa$ to $2^\kappa$. It is well known that at regular cardinals the continuum function is very easily changed by forcing, as was shown by Easton \cite{EASTONregular}. The case of singular cardinals, or regular limit cardinals whose ``largeness'' we wish to preserve, is more difficult and gave rise to several results which generalize Easton's theorem in this direction (see for instance \cite{MENsup}, \cite{RADEKeaston}, \cite{Cod:Wood}
 or \cite{CodMag:super}).

In this paper we study yet another generalization of Easton's theorem in which we require that some successor cardinals should retain their largeness in terms of a certain compactness property. If $\lambda$ is a  regular uncountable cardinal,  we say that $\lambda$ has \emph{the tree property}, and we denote it by $\TP(\lambda)$, if all $\lambda$-trees have a cofinal branch. It is known that if the tree property holds at $\kappa^{++}$, then $2^\kappa > \kappa^{+}$. In other words the tree property has a non-trivial effect on the continuum function. It seems natural to ask whether the tree property at $\kappa^{++}$ puts more restrictions on the continuum function in addition to $2^\kappa>\kappa^+$ (and the usual restrictions which the continuum function needs to satisfy); or equivalently, which continuum functions are compatible with the tree property. Since it is still open how to get the tree property at a long interval of cardinals (for more information see \cite{NEEMAN:tp}), any Easton's theorem for the tree property is at the moment limited to countable intervals of cardinals.

As should be expected, the difficulty of this question increases if we wish to have (A) the tree property at consecutive cardinals or (B) at cardinals which are the successors or double successors of singular cardinals. We deal with the type (A) in this paper. 

The first partial answer to (A) was given by Unger (\cite{UNGER:1}) who showed that the tree property at $\aleph_2$ is consistent with $2^{\aleph_0}$ arbitrarily large.\footnote{The result can be easily generalized to an arbitrary regular cardinal $\kappa$ with the tree property at $\kappa^{++}$.} We generalized this result in \cite{HS:tp} for all cardinals below $\aleph_\omega$ for the weak tree property (no special Aronszajn trees) and for all even cardinals $\aleph_{2n}$ for the full tree property. The argument used infinitely many weakly compact cardinals which is optimal for the result. In \cite{HS:tp}, we left open the natural question whether having the tree property at every $\aleph_n$ for $2 \le n < \omega$ is consistent with any continuum function which violates GCH below $\aleph_\omega$. Unlike the argument in \cite{HS:tp}, this requires much larger cardinals because it is known that consecutive cardinals with the tree property imply the consistency of at least a Woodin cardinal (see \cite{lower}). In this paper we provide the affirmative answer to this question, i.e.\ we show that if there are infinitely many supercompact cardinals, then it is consistent that the tree property holds at every $\aleph_n$ for $2 \le n < \omega$, and the continuum function below $\aleph_\omega$ is anything not outright inconsistent with the tree property.\footnote{There is nothing specific about the $\aleph_n$'s; the final consecutive sequence $\seq{\kappa_n}{n<\omega}$ of regular cardinals with the tree property can live much higher.}

The argument is based on the construction in the paper by Cummings and Foreman \cite{CUMFOR:tp}, extended to obtain the right continuum function. We outline the argument in Section \ref{sec:outline}.

Although it is not the focus of this paper, let us say a few words about the type (B). We showed in \cite{FHS1:large} that the tree property at the double successor of a singular strong limit cardinal $\kappa$ with countable cofinality does not put any restrictions on the value of $2^\kappa$ apart from the trivial ones.\footnote{An easier proof of this theorem can be found in \cite{HS2:Indestructibility}; the proof is based on an application of the indestructibility of the tree property under certain $\kappa^+$-cc forcing notions. The advantage of the new proof is that it can be directly generalized to singular cardinals with an uncountable cofinality (it does not use any of the properties of the Prikry-type forcing notions except the chain condition).} In \cite{FHS2:large} we followed up with the result that $2^{\aleph_\omega}$ can be equal to $\aleph_{\omega+2+n}$ for any $n< \omega$ with the tree property holding at $\aleph_{\omega+2}$.

\subsection{An outline of the argument}\label{sec:outline}
Let us briefly outline the structure of the argument for a reader roughly familiar with the papers of Abraham \cite{ABR:tree} and Cummings and Foreman \cite{CUMFOR:tp}. Let $\kappa_n$, $1 < n<\omega$, be an increasing sequence of supercompact cardinals with $\kappa_0 = \aleph_0$ and $\kappa_1 = \aleph_1$. For forcing the tree property at $\kappa = \kappa_{n+2}$ for $n \ge 1$ over some model $V_{n-1}$, we are going to use a variant of the Mitchell forcing as it was defined in \cite{CUMFOR:tp}; this forcing contains the Cohen forcing at $\kappa_{n}$. If we define this Cohen forcing in $V_{n-1}$, $V_{n-1}$ must satisfy $\kappa_{n}^{<\kappa_{n}} = \kappa_{n}$ otherwise some cardinals above $\kappa_{n}$ will be unintentionally collapsed. $\kappa_{n}$ is either $\omega_1$ or an inaccessible cardinal in the ground model $V$, but in either case it will be a successor cardinal in $V_{n-1}$, in fact it will be the successor of $\kappa_{n-1}$ (more to the point, it will be the $\aleph_{n}$ of $V_{n-1}$). It follows that for forcing the tree property at $\kappa$ over $V_{n-1}$, the Cohen forcing at $\kappa_{n}$ must come from a model where $2^{\kappa_{n-1}} \le \kappa_n$. Since by the inductive construction for the tree property we will necessarily have $2^{\kappa_{n-1}} > \kappa_{n}$ in $V_{n-1}$, the Cohen forcing cannot come from $V_{n-1}$, but should come from some earlier model.\footnote{We should add that this implies that the Cohen forcing will no longer be $\kappa_{n}$-closed in $V_{n-1}$ so an additional argument must be provided for not collapsing below $\kappa_{n}$.} Cummings and 
Foreman solved this problem by postulating the the Cohen forcing at $\kappa_{n}$ comes from the model $V_{n-2}$, which works provided that $2^{\kappa_{n-1}} = \kappa_{n}$ in $V_{n-2}$. Unless we manipulate the continuum function further, this will leave us with gap $2$ below $\aleph_\omega$: $2^{\aleph_n} = \aleph_{n+2}$ for all $n<\omega$.

In order to realize an arbitrary Easton function below $\aleph_\omega$ (which satisfies $2^{\aleph_n} \ge \aleph_{n+2}$ for all $n<\omega)$ we need to modify the construction of Cummings and Foreman in some way. There seem to be essentially two options: (i) modify the construction in Cummings and Foreman directly and add the required number of subsets of $\kappa_n$ by a Cohen forcing which lives in $V_{n-2}$, or in some earlier model, perhaps even the ground model $V$, or (ii) leave the inductive construction for the tree property as it is in Cummings and Foreman (which gives gap $2$ for the continuum function) and increase the powersets as required in the next step.

The option (i) mays seem cleaner at the first sight, but it causes technical complications\footnote{Roughly speaking, it is hard to argue for the distributivity of the tail of the Mitchell iteration (i.e.\ a tail of $\R_\omega$ in (\ref{R_omega})). In option (ii), the distributivity is ensured by closure in a suitable submodel (essentially an application of Easton's lemma).} because both tasks -- ensuring the tree property and the right continuum function -- are mixed into a single iteration. The option (ii) deals with the two tasks separately, but one needs to make sure that forcing the right continuum function does not ``undo'' the tree property part.

We have opted for the option (ii) and defined a certain forcing $\bb{Z}$ so that \beq \label{R_omega} \bb{Z} = \R_\omega * \dot{\E},\eeq where $\R_\omega$ is exactly the forcing from Cummings and Foreman paper and $\dot{\E}$ is a full-support product of Cohen forcings to obtain the desired continuum function. The Cohen forcings in $\dot{\E}$ are chosen from appropriate inner models of the extension $V[\R_\omega]$ in order to satisfy the restrictions described in previous paragraphs (more precisely, the Cohen forcing at some $\kappa_n$ in $\dot{\E}$ comes from the same inner model as the Cohen at $\kappa_n$ which is the part of the Mitchell forcing in the iteration $\R_\omega$).

The present paper is structured as follows. In Section \ref{sec:prelim} we provide some background information to make the paper self-contained. First we review some basic forcing properties which deal with the interactions of the chain condition and the closure between different models (Section \ref{Sec:FP}), then we discuss forcing conditions for not adding cofinal branches to certain trees (Section \ref{sec:tf}), and finally we review the Mitchell forcing and the argument of Cummings and Foreman from \cite{CUMFOR:tp}.

In Section \ref{sec:main} we prove our theorem. The argument is divided into two sections: In Section \ref{sub:contf} we show that the forcing $\bb{Z}$ collapses only the intended cardinals and moreover forces the right continuum function. In Section \ref{sub:tp} we show that $\bb{Z}$ forces the tree property at every $\aleph_n$, $2 \le n < \omega$, which finishes the argument.

In the final section we discuss open questions and further research.

\section{Preliminaries}\label{sec:prelim}
\subsection{Some basic properties of forcing notions}\label{Sec:FP}

In this section we review some basic properties which we will use later in the paper.

\begin{definition}\label{Def:Forcing_Properties}
Let $P$ be a forcing notion and let $\kappa>\aleph_0$ be a regular cardinal. We say that $P$ is: 
\begin{itemize}
\item \emph{$\kappa$-cc} if every antichain of $P$ has size less than $\kappa$ (we say that $P$ is $ccc$ if it is $\aleph_1$-cc).
\item \emph{$\kappa$-Knaster} if for every $X\sub P$ with $|X|=\kappa$ there is $Y\subseteq X$, such that $|Y|=\kappa$ and all elements of $Y$ are pairwise compatible.
\item \emph{$\kappa$-closed} if every decreasing sequence of conditions in $P$ of size less than $\kappa$ has a lower bound.
\item \emph{$\kappa$-distributive} if $P$ does not add new sequences of ordinals of length less than $\kappa$. 
\end{itemize}
\end{definition}

It is easy to check that all these properties -- except for the $\kappa$-closure -- are invariant under forcing equivalence\footnote{We say that $(P,\le_P)$ and $(Q,\le_Q)$ are forcing equivalent if their Boolean completions are isomorphic.}. Regarding the closure, note that for every non-trivial forcing notion $P$ which is $\kappa$-closed there exists a forcing-equivalent forcing notion which is not even $\aleph_1$-closed (e.g.\ the Boolean completion of $P$).

\begin{lemma}\label{L:iterace_property}
Let $\kappa>\aleph_0$ be a regular cardinal and assume that $P$ is a forcing notion and $\dot{Q}$ is a $P$-name for a forcing notion. Then the following hold: 
\bce[(i)]
\item $P$ is $\kappa$-closed and $P$ forces $\dot{Q}$ is $\kappa$-closed if and only if $P*\dot{Q}$ is $\kappa$-closed. 
\item $P$ is $\kappa$-distributive and $P$ forces $\dot{Q}$ is $\kappa$-distributive if and only if $P*\dot{Q}$ is $\kappa$-distributive.
\item $P$ is $\kappa$-cc and $P$ forces $\dot{Q}$ is $\kappa$-cc if and only if $P*\dot{Q}$ is $\kappa$-cc.
\item If $P$ is $\kappa$-Knaster and $P$ forces $\dot{Q}$ $\kappa$-Knaster then $P*\dot{Q}$ is $\kappa$-Knaster
\ece
\end{lemma}

\begin{proof}
The proofs are routine; for more details see \cite{JECHbook} or \cite{KUNbook}.
\end{proof}

If $Q$ is in the ground model, $P * \check{Q}$ is equivalent to $P \x Q$.  We state some properties which the product forcing has with respect to the chain condition.

{~}
\begin{lemma}\label{lm:product}
Let $\kappa>\aleph_0$ be a regular cardinal and assume that $P$ and $Q$ are forcing notions. Then the following hold: 
\bce[(i)]
\item If $P$ and $Q$ are $\kappa$-Knaster, then $P \x Q$ is $\kappa$-Knaster.
\item If $P$ is $\kappa$-Knaster and $Q$ is $\kappa$-cc, then $P\x Q$ is $\kappa$-cc.
\ece
\end{lemma}

\begin{proof}
The proofs are routine.
\end{proof}

The following lemma summarises some of the more important forcing properties of a product $P \x Q$ regarding the chain condition.

{~}
\begin{lemma}\label{L:K-cc}
Let $\kappa>\aleph_0$ be a regular cardinal and assume that $P$ and $Q$ are forcing notions such that $P$ is $\kappa$-Knaster and $Q$ is $\kappa$-cc. Then the following hold:
\bce[(i)]
\item $P$ forces that $Q$ is $\kappa$-cc.
\item $Q$ forces that $P$ is $\kappa$-Knaster.
\ece
\end{lemma}
\begin{proof}
(i).\ This is an easy consequence of Lemmas \ref{L:iterace_property}(iii) and \ref{lm:product}(ii).

(ii).\ A proof (attributed to Magidor) can be found in \cite{CumAron-Kurep}. 
\end{proof}

The following lemma summarises some of the more important properties of the product $P \x Q$ regarding the distributivity and closure. 

\begin{lemma}\label{L:closed-dis}
Let $\kappa>\aleph_0$ be a regular cardinal and assume that $P$ and $Q$ are forcing notions, where $P$ is $\kappa$-closed and $Q$ is $\kappa$-distributive.  Then the following hold:
\bce[(i)]
\item $P$ forces that $Q$ is $\kappa$-distributive.
\item $Q$ forces that $P$ is $\kappa$-closed.
\ece
\end{lemma}

\begin{proof}
The proof is routine.
\end{proof}

We can also formulate some results for the product of two forcing notions with respect to preservation of the chain condition and distributivity at the same time. The following lemma appeared in \cite{EASTONregular}.

\begin{lemma}\label{L:Easton}\emph{(Easton)}
Let $\kappa>\aleph_0$ be a regular cardinal and assume that $P$ and $Q$ are forcing notions, where $P$ is $\kappa$-cc and $Q$ is $\kappa$-closed. Then the following hold:
\bce[(i)]
\item $P$ forces that $Q$ is $\kappa$-distributive.
\item $Q$ forces that $P$ is $\kappa$-cc.
\ece
\end{lemma}

\begin{proof}
For the proof of (i), see \cite[Lemma 15.19]{JECHbook}, (ii) is easy.
\end{proof}

\subsection{Trees and forcing}\label{sec:tf}

An essential step in standard arguments that a certain partial order forces the tree property is to argue that its quotient does not add cofinal branches to certain trees. Fact \ref{F:Knaster} is due to Baumgartner (see \cite{BAUM:iter}) and Fact \ref{F:Closed} is due to Silver (see \cite{ABR:tree} for more details; a proof with $\lambda = \aleph_0$ is in \cite[Chapter VIII, Section 3]{KUNbook}) .
 
{~}
\begin{fact}\label{F:Knaster}
Let $\kappa$ be a regular cardinal and assume that $P$ is a $\kappa$-Knaster forcing notion. If $T$ is a tree of height $\kappa$, then forcing with $P$ does not add cofinal branches to $T$.
\end{fact}

\begin{fact}\label{F:Closed}
Let $\kappa$, $\lambda$ be regular cardinals and $2^{\kappa}\geq\lambda$. Assume that $P$ is a $\kappa^{+}$-closed forcing notion. If $T$ is a $\lambda$-tree, then forcing with $P$ does not add cofinal branches to $T$.
\end{fact}

These facts can be generalized as follows (for the first fact see \cite{Unger:succ}; the first statement of the second fact appeared in \cite{JSkurepa} with $\kappa=\aleph_0$ and $\lambda=\aleph_1$, the general version is due to Unger in \cite{UNGER:1}).

\begin{fact}\label{F:Square_cc}
Let $\kappa$ be a regular cardinal and assume that $P$ is a forcing notion such that square of $P$, $P\x P$, is $\kappa$-cc. If $T$ is a tree of height $\kappa$, then forcing with $P$ does not add cofinal branches to $T$.
\end{fact}

\begin{fact}\label{F:ccc_Closed}
Let $\kappa < \lambda$ be regular cardinals and $2^{\kappa}\geq\lambda$. Assume that $P$ and $Q$ are forcing notions such that $P$ is $\kappa^{+}$-cc and $Q$ is $\kappa^{+}$-closed. If $T$ is a $\lambda$-tree in $V[P]$, then forcing with $Q$ over $V[P]$ does not add cofinal branches to $T$.
\end{fact}

\subsection{Mitchell forcing}\label{S:Mitchell}

Mitchell forcing was defined by Mitchell in \cite{M:tree}. In this section we review several variants of the Mitchell forcing, which can be found in papers \cite{ABR:tree} and \cite{CUMFOR:tp}. All proofs of facts stated below can be found in these papers as well. If $\kappa$ is a regular cardinal and $\alpha$ a limit ordinal, let $\Add(\kappa,\alpha)$ be the set of all partial functions of size $<\kappa$ from $\mx{SuccOrd}(\alpha)$ to $2$, ordered by reverse inclusion, where $\mx{SuccOrd}(\alpha)$ is the set of all successor ordinals below $\alpha$.\footnote{This is just a technical assumption which will be useful in analysis of Mitchell forcing. See paragraph below Remark \ref{T_Collaps}.} It is easy to see that this forcing is isomorphic to the usual Cohen forcing for adding $\alpha$-many subsets of $\kappa$. It follows that if $\beta < \alpha$ and $p \in \Add(\kappa,\alpha)$, then $\restr{p}{\beta}$ is in $\Add(\kappa,\beta)$.

\begin{definition}\label{Def:Mitchell}
Let $\kappa$ be a regular cardinal and $\lambda>\kappa$ an inaccessible cardinal. The \emph{Mitchell forcing} at $\kappa$ of length $\lambda$, denoted by $\M(\kappa,\lambda)$, is the set of all pairs $(p,q)$ such that $p$ is in Cohen forcing $\Add(\kappa,\lambda)$ and $q$ is a function with $\dom{q}\sub\lambda$ of size at most $\kappa$ and for every $\alpha\in\dom{q}$, $\alpha$ is a successor cardinal and it holds:
\beq
1_{\Add(\kappa,\alpha)}\Vdash q(\alpha)\in\dot{\Add}(\kappa^+,1),
\eeq 
where $\dot{\Add}(\kappa^+,1)$ is the canonical $\Add(\kappa,\alpha)$-name for Cohen forcing at $\kappa^+$. 
A condition $(p,q)$ is stronger than $(p',q')$ if
\bce[(i)]
\item $p\leq p'$,
\item $\dom{q}\supseteq \dom{q'}$ and for every $\alpha\in\dom{q'}$, ${p}\rest{\alpha}\Vdash q(\alpha)\leq q'(\alpha)$.
\ece
\end{definition}

Assuming that $\kappa<\lambda$, $\kappa$ is regular, and $\lambda$ is inaccessible, Mitchell forcing $\M(\kappa,\lambda)$ is $\lambda$-Knaster and $\kappa$-closed. Moreover if $\kappa^{<\kappa}=\kappa$, $\M(\kappa,\lambda)$ preserves $\kappa^+$ (by a product analysis of Abraham  \cite{ABR:tree}), collapses cardinals exactly in the open interval $(\kappa^+,\lambda)$ and forces $2^\kappa=\lambda = \kappa^{++}$.

\begin{theorem}\emph{(Mitchell)} Assume $\kappa^{<\kappa} = \kappa$. If $\lambda$ is a weakly compact cardinal, then $\M(\kappa,\lambda)$ forces the tree property at $\lambda = \kappa^{++}$.
\end{theorem}

We modify the definition of Mitchell forcing in two steps. In the first step we define the variation of Mitchell forcing where the Cohen part of Mitchell forcing is taken from some suitable inner model of our universe. In the second step we add a third coordinate which will prepare the universe for a further lifting of an appropriate embedding.

\begin{definition}\label{Def_Mitchell_2Models}
Let $V\sub W$ be two inner models of $\ZFC$ with the same ordinals, $\kappa$ be a regular cardinal and $\lambda>\kappa$ inaccessible in $W$. Suppose that $\Add(\kappa,\lambda)^V$ is in $W$ $\kappa^+$-cc and $\kappa$-distributive. In $W$, the \emph{Mitchell forcing} at $\kappa$ of length $\lambda$, denoted by $\M(\kappa,\lambda, V,W)$, is the set of all pairs $(p,q)$, where $p$ is a condition in $\Add(\kappa,\lambda)^V$ and $q$ is a function in $W$ such that $\dom{q}$ is a subset of open interval $(\kappa,\lambda)$ of size at most $\kappa$ and for every $\alpha\in\dom{q}$, $\alpha$ is a successor cardinal and the following holds:
\beq
1_{\Add(\kappa,\alpha)^V} \Vdash^W q(\alpha)\in\dot{\Add}(\kappa^+,1)^W,
\eeq
where $\dot{\Add}(\kappa^+,1)^W$ is $\Add(\kappa,\alpha)^V$-name for the Cohen forcing at $\kappa^+$ over the model $W$. The ordering is defined by $(p,q)\leq (p',q')$ if
\bce[(i)]
\item $p\leq p'$,
\item $\dom{q}\supseteq \dom{q'}$ and for every $\alpha\in\dom{q'}$, ${p}\rest{\alpha}\Vdash q(\alpha)\leq q'(\alpha)$.
\ece
\end{definition}

Now we review the original forcing which iteration was used to force the tree property below $\aleph_\omega$. For more details see \cite{CUMFOR:tp} and \cite{ABR:tree}.

\begin{fact}
Let $\lambda$ be a supercompact cardinal. Then there is a function $F$ from $\lambda$ to $V_{\lambda}$ such that for  all $\mu\ge\lambda$ and all $x\in H_{\mu^+}$ there is a supercompactness measure $U$ on $\mathscr{P}_\lambda(\mu)$ such that $j_U(F)(\lambda)=x$. We call $F$ a \emph{Laver function} for $\lambda$.
\end{fact}

Let $F_\lambda:\lambda\then V_\lambda$ denote a Laver function from previous fact for a given supercompact cardinal $\lambda$.

\begin{definition}
Let $V\sub W$ be two inner models of $\ZFC$ with the same cardinals, $\kappa$ be a regular cardinal and $\lambda>\kappa$ supercompact in $W$. Suppose that $\Add(\kappa,\lambda)^V$ is $\kappa^+$-cc and $\kappa$-distributive in $W$. The forcing $\R(\kappa,\lambda, V, W, F_\lambda)$ is the set of all triples $(p,q,f)$ such that $(p,q)$ is in the Mitchell forcing $\M(\kappa,\lambda, V,W)$ and $f$ is a function in $W$ of size less than $\kappa^+$ 
such that $\dom{f}$ is a subset of 
\begin{multline}
\{\alpha<\lambda \; | \;\alpha\mbox{ inaccessible and }1_{\R|\alpha} \Vdash^W F_{\lambda}(\alpha) \mbox{ is an }\alpha\mbox{-directed closed forcing}\},
\end{multline}
and if $\alpha\in\dom{f}$ then $f(\alpha)\in W^{\R|\alpha}$ and $1_{\R|\alpha}\Vdash^W f(\alpha)\in F_\lambda(\alpha)$.

The ordering is defined by $(p,q,f)\leq (p',q',f')$ if
\bce[(i)]
\item $(p,q)\leq (p',q')$,
\item $\dom{f}\supseteq \dom{f'}$ and for every $\alpha\in\dom{f'}$, $(p\rest{\alpha},q\rest{\alpha},f\rest{\alpha})\Vdash f(\alpha)\leq f'(\alpha)$.
\ece
\end{definition}

Note that the previous definition should be formally defined by induction on $\lambda$, for more details see \cite{CUMFOR:tp}. Also note that the definition is made in the model $W$ and all what we are state in further is in sense of the model $W$.

Mitchell forcing $\R(\kappa,\lambda, V,W,F_\lambda)$ is $\lambda$-Knaster and $\kappa$-distributive. Moreover, it collapses the cardinals in the open interval $(\kappa^+,\lambda)$ to $\kappa^+$ and forces $2^\kappa=\lambda = \kappa^{++}$. The preservation of $\kappa^+$ is shown by means of the product analysis due to Abraham \cite{ABR:tree}.

Let $\T$ be defined as follows:
\beq
\T=\set{(\emptyset,q,f)}{(\emptyset,q,f)\in \R(\kappa,\lambda,V,W,F_\lambda)}.
\eeq
The ordering on $\T$ is the one induced from $ \R(\kappa,\lambda,V,W,F_\lambda)$. It is clear that $\T$ is $\kappa^+$-directed closed in $W$. We will call $\T$ the \emph{term forcing} (of the associated Mitchell-style forcing).

It is easy to see that the function \beq \label{eq:pi-M}\pi:\Add(\kappa,\lambda)^V\x\T\then \R(\kappa,\lambda,V,W,F_\lambda)\eeq which maps  $(p,(\emptyset,q,f))$ to $(p,q,f)$ is a projection. Since the product $\Add(\kappa,\lambda)^V \x \T$ preserves $\kappa^+$ (under assumption $\kappa^{<\kappa}=\kappa$), so does the forcing $\R(\kappa,\lambda,V,W,F_\lambda)$.

There are natural projections from Mitchell forcing of length $\lambda$ to Mitchell forcings of shorter lengths and a projection to Cohen forcing $\Add(\kappa,\lambda)^V$. For the first claim, define a function $\sigma^{\lambda,\alpha}$ from $\R(\kappa,\lambda,V,W,F_\lambda)$ to $\R(\kappa,\alpha,V,W,F_\lambda)$, where $\alpha$ is an ordinal between $\kappa$ and $\lambda$, as follows: $\sigma^{\lambda,\alpha}((p,q,f))=({p}\rest{\alpha},{q}\rest{\alpha},f\rest\alpha)$. For the second claim, define a function $\rho$ from $\R(\kappa,\lambda,V,W,F_\lambda)$ to $\Add(\kappa,\lambda)^V$ by $\rho((p,q,f))=p$. It is easy to see that $\sigma^{\lambda,\alpha}$ and $\rho$ are projections.

By the projection $\rho:\R(\kappa,\lambda,V,W,F_\lambda)\to \Add(\kappa,\lambda)^V$, $\R(\kappa,\lambda,V,W,F_\lambda)$ is forcing-equivalent to $\Add(\kappa,\lambda)^V*\dot{\D}$, for some $\dot{\D}$. Moreover, by the product analysis  (i.e.\ of the existence of the projection $\pi$), $\dot{\D}$ is a name for a forcing notion which is forced to be $\kappa^+$-distributive and $\kappa$-closed.

\begin{remark}\label{T_Collaps}
\emph{Notice that the term forcing $\T$ collapses the cardinals between $\kappa^+$ and $\lambda$: Suppose $\kappa^{<\kappa} = \kappa$ and $\lambda$ is inaccessible. As $\T$ is $\kappa^+$-closed, Cohen forcing $\Add(\kappa,\lambda)$ is still $\kappa^{+}$-cc and $\kappa$-closed in $V[\T]$. In particular, it does not collapse cardinals over $V[\T]$ (so it must be $\T$ which collapses the cardinals).}
\end{remark}

The term forcing analysis carries over to quotients given by the projections $\sigma^{\lambda,\alpha}$ whenever $\alpha$ is an inaccessible cardinal between $\kappa$ and $\lambda$. First note that if $\alpha$ is inaccessible then $\R(\kappa,\alpha+1,V,W,F_\lambda)$ is equivalent to $\R(\kappa,\alpha,V,W,F_\lambda)*F(\alpha)$. This holds because at limit cardinals the first coordinates are not defined.

Let $G_{\alpha+1}$ be an $\R(\kappa,\alpha+1,V,W,F_\lambda)$-generic filter and define in $V[G_{\alpha+1}]$ the quotient $\R(\kappa,\lambda,V,W,F_\lambda)/G_{\alpha+1}$ as follows:
\begin{multline}
\R(\kappa,\lambda,V,W,F_\lambda)/G_{\alpha+1}= \\ \set{(p,q,f)\in \R(\kappa,\lambda,V,W,F_\lambda)}{({p}\rest{\alpha},{q}\rest{\alpha},f\rest\alpha+1)\in G_{\alpha+1}}. 
\end{multline}
Regarding this quotient, we can now analogously define the term forcing $\T^*$ in $V[G_{\alpha+1}]$
\beq 
\T^*=\set{(\emptyset,q,f)}{(\emptyset,q,f)\in \R(\kappa,\lambda,V,W,F_\lambda)/G_{\alpha+1}}.
\eeq
and a projection $\pi^*$ from $\Add(\kappa,\lambda-\alpha)\x \T^*$ to $\R(\kappa,\lambda,V,W,F_\lambda)/G_{\alpha+1}$ by setting $\pi^*((p,(\emptyset,q,f)))=(p,q,f)$.

\begin{fact}\label{F:projekce_quotient}
 Let $\alpha$ be inaccessible and $G_{\alpha+1}$ an $\R(\kappa,\alpha+1,V,W,F_\lambda)$-generic filter. Then in $V[G_{\alpha+1}]$ the following hold:
\bce[(i)]
\item $\pi^*$ is a projection from $\Add(\kappa,\lambda-\alpha)\x\T^*$ to $\R(\kappa,\lambda,V,W,F_\lambda)/G_{\alpha+1}$.
\item $\T^*$ is $\kappa^+$-closed in $V[G_{\alpha+1}]$.
\ece
\end{fact}

At the end of the analysis, consider the quotient of $\Add(\kappa,\lambda)\x\T$ after the forcing $\R(\kappa,\lambda,V,W,F_\lambda)$. Let $G$ be $\R(\kappa,\lambda,V,W,F_\lambda)$-generic. We define
\beq 
\S=(\Add(\kappa,\lambda)\x\T)/G=\set{(p,(\emptyset,q,f))\in \Add(\kappa,\lambda)\x\T}{(p,q,f)\in G}.
\eeq

\begin{fact}
$\S$ is $\kappa$-closed, $\kappa^+$-distributive and $\lambda$-cc over $V[\R]$.
\end{fact}

The following lemma summarises properties which are preserved after forcing with a product of a Mitchell-style forcing and another forcing.

\begin{lemma}\label{TPd}
Let $V\sub W$ be two inner models of $\ZFC$ with the same cardinals, $\kappa$ be a regular cardinal and $\lambda>\kappa$ supercompact in $W$. Suppose that $\Add(\kappa,\lambda)^V$ is $\kappa^+$-Knaster and $\kappa$-distributive in $W$. Assume $P$ is $\kappa^+$-Knaster, $R$ is $\kappa^+$-cc and $Q$ and $S$ are $\kappa^+$-closed in $W$. Then the following hold:
\bce[(i)]
\item $R \x \R(\kappa,\lambda,V,W,F_\lambda)$ forces that $Q$ is $\kappa^+$-distributive.
\item $Q \x \R(\kappa,\lambda,V,W,F_\lambda)$ forces that $R$ is $\kappa^+$-cc.
\item $P \x \R(\kappa,\lambda,V,W,F_\lambda)$ forces that $R$ is $\kappa^+$-cc.
\item $Q \x \R(\kappa,\lambda,V,W,F_\lambda)$ forces that $S$ is $\kappa^+$-distributive.
\ece
\end{lemma}
\begin{proof}
(i).\ It is easy to check that the projection $\pi$ in (\ref{eq:pi-M}) extends to the projection $\pi'$, \beq \label{pi'} \pi': R \x Q \x \Add(\kappa,\lambda) \x \T \to R \x Q \x \R(\kappa,\lambda,V,W,F_\lambda),\eeq which sends $(r_1, r_2, p, (\emptyset, q,f))$ to $(r_1, r_2, (p,q,f))$.
It follows that $R \x Q \x \Add(\kappa,\lambda)^V \x \T$ is forcing equivalent to \beq \label{S}[R \x Q \x \R(\kappa,\lambda,V,W,F_\lambda)] * \dot{S}\eeq for some quotient forcing $\dot{S}$.

Let $G \x g \x F$ be an arbitrary $R \x \R(\kappa,\lambda,V,W,F_\lambda)\x Q$-generic filter over $W$. We will show that every sequence $x$ of ordinals of length less than $\kappa^+$ which is in $V[G\x g \x F]$ is in $V[G \x g]$ which shows that $Q$ is forced to be $\kappa^+$-distributive as required.
Let $x$ as above be fixed. Let $h$ be any $\dot{S}$-generic filter over $V[G \x g \x F]$. It follows by (\ref{S}) that $V[G \x g \x F][h]$ can be written as $V[G \x g_0 \x g_1 \x F]$ where $g_0 \x g_1$ is $\Add(\kappa,\lambda)^V \x \T$-generic, and the following hold:
\bce[(i)]
\item $V[G \x g \x F] \sub V[G \x g_0 \x g_1 \x F]$,
\item $V[G \x g_0] \sub V[G \x g]$,
\ece
where (ii) holds because $g_0$ is the Cohen part of $g$. In particular $x$ is in $V[G \x g_0 \x g_1 \x F]$.

By Easton's lemma, $R \x \Add(\kappa,\lambda)^V$ (which is $\kappa^+$-cc) forces that $\T\x Q$ (which is $\kappa^+$-closed) is $\kappa^+$-distributive. It follows that $x$ is already in $V[G \x g_0]$, and hence in $V[G \x g]$ as desired.

(ii) -- (iv).\ It suffices to argue similarly as in (i) that the forcing notion under consideration has the required property in the generic extension by $Q \x \Add(\kappa,\lambda)^V \x \T$ for (ii) and (iv), and $P \x \Add(\kappa,\lambda)^V \x \T$ for (iii). This is easy to show using the Easton's lemma (Lemma \ref{L:Easton}).
\end{proof}

\subsection{The Cummings-Foreman model}

Let $\kappa_2<\kappa_3<\dots$ be an $\omega$-sequence of supercompact cardinals with limit $\lambda$ and let $\kappa_0$ denote $\aleph_0$ and $\kappa_1$ denote $\aleph_1$. And let $F_n$ denote corresponding Laver function for $\kappa_n$ for $n>1$. Now we define Cummings-Foreman forcing used in \cite{CUMFOR:tp} to force the tree property below $\aleph_\omega$. We also state some basic facts about this forcing which can be found in \cite{CUMFOR:tp}.

\begin{definition}
The iteration $\R_\omega=\seq{\R_n*\dot{\Q}_n}{n<\omega}$ of length $\omega$ is defined by induction as follows:
\bce[(i)]
\item The first stage $\Q_0=\R(\kappa_0,\kappa_2,V,V, F_2)$, let us denote $\R_1=\Q_0$ and $\R_0$ be the trivial forcing.
\item Suppose that we have defined the iteration up to stage $n>0$. Let $\R_n=\Q_0*\dots *\dot{\Q}_{n-1}$. First define an $\R_n$-name $\dot{F}_{n+2}$ by $\dot{F}_{n+2}(\alpha)=F_{n+2}(\alpha)$, if $F_{n+2}(\alpha)$ is an $\R_n$-name, and $\dot{F}_{n+2}(\alpha)=0$ otherwise. Then define $\dot{\Q}_n$ to be a name for $\R(\kappa_n,\kappa_{n+2},V[\R_{n-1}],V[\R_n], F^*_{n+2})$, where $F^*_{n+2}$ is the interpretation of $\dot{F}_{n+2}$ in $V[\R_n]$.
\ece
Let $\R_\omega$ denote the inverse limit of $\seq{\R_n}{n<\omega}$.
\end{definition}

Let us for $n<\omega$ fix the following notation corresponding to the analysis in the previous section. Let $\T_n$, $\D_n$ and $\S_n$ be the relevant partial orders and $\pi_n$, $\rho_n$ and $\sigma^{\kappa_{n+2},\alpha}$ the projections, where $\alpha$ is an ordinal between $\kappa_n$ and $\kappa_{n+2}$.

For the proofs of the following facts see corresponding lemmas in \cite{CUMFOR:tp} (Lemma 4.2, Lemma 4.3 and Lemma 4.4).

\begin{fact}
Let $\P_0$ denote $\Add(\kappa_0,\kappa_2)$ and $\T_0$ denote the term forcing of by $\Q_0=\R(\kappa_0,\kappa_2,V,V,F_2)$. Then the following hold:
\bce[(i)]
\item The size of $\Q_0$ is $\kappa_2$ and $\Q_0$ is $\kappa_2$-Knaster.
\item $\pi_0$ is a projection from $\P_0\times \T_0$ to $\Q_0$ and $\rho_0$ is a projection from $\Q_0$ to $\P_0$.
\item $\Q_0$ forces $2^{\aleph_0}=\kappa_2=\aleph_2$.
\item $\Add(\kappa_1,\xi)^V$ is $\kappa_1$-distributive and $\kappa_2$-Knaster after forcing with $\Q_0$ for a suitable ordinal $\xi>0$.
\item $\dot{\D}_0$, given by the projection $\rho_0$, is a $\P_0$-name for $\kappa_1$-distributive and $\kappa_2$-cc forcing.
\item $\dot{\S}_0$, given by the projection $\pi_0$ is a $\Q_0$-name for $\kappa_1$-distributive and $\kappa_2$-cc forcing.
\ece
\end{fact}

\begin{fact}\label{F:iterace}
Let $n>0$ and let us denote by $\P_n=\Add(\kappa_n,\kappa_{n+2})^{V[\R_{n-1}]}$ and $\T_n$ the term forcing of $\Q_n=(\kappa_n,\kappa_{n+2},V[\R_{n-1}],V[\R_n],F^*_{n+2})$. Then in $V[\R_n]$ the following hold:
\bce[(i)]
\item $2^{\kappa_i}=\kappa_{i+2}$ for $i<n$ and $\kappa_i=\aleph_i$ for $i<n+2$.
\item The size of $\Q_n$ is $\kappa_{n+2}$ and $\Q_n$ is $\kappa_{n-1}$-closed, $\kappa_n$-distributive and $\kappa_{n+2}$-Knaster.
\item $\Q_n$ is a projection of $\P_n\times\T_n$ and there is also projection from $\Q_n$ to $\P_n$.
\item $\Q_n$ forces $2^{\kappa_n}=\kappa_{n+2}=\aleph_{n+2}$.
\item $\Add(\kappa_{n+1},\xi)^{V[\R_n]}$ is $\kappa_{n+1}$-distributive and $\kappa_{n+2}$-Knaster after forcing with $\Q_n$ a suitable ordinal $\xi>0$.
\item $\dot{\D}_n$, given by the projection $\rho_n$, is a $\P_n$-name for $\kappa_n$-closed, $\kappa_{n+1}$-distributive and $\kappa_{n+2}$-cc forcing.
\item $\dot{\S}_n$, given by the projection $\pi_n$ is a $\Q_n$-name for $\kappa_n$-closed, $\kappa_{n+1}$-distributive and $\kappa_{n+2}$-cc forcing.
\ece
\end{fact}

\begin{fact}
Let $n\ge 0$. Any $\kappa_n$-sequence of ordinals in $V[\R_\omega]$ is already added by $\R_n*\dot{\P}_n$.
\end{fact}

\begin{theorem}\emph{(Cummings-Foreman)}
In the generic extension by $\R_\omega$ the following hold:
\bce[(i)]
\item $2^{\kappa_n}=\kappa_
{n+2}$ and $\kappa_n=\aleph_n$, for $n<\omega$,
\item the tree property at $\kappa_n$, for $1<n<\omega$.
\ece 
\end{theorem}

\section{Main theorem}\label{sec:main}

Let $\kappa_2<\kappa_3<\dots$ be an $\omega$-sequence of supercompact cardinals with limit $\lambda$ and let $\kappa_0$ denote $\aleph_0$ and $\kappa_1$ denote $\aleph_1$. In Theorem \ref{Th:main}, we control the continuum function below $\aleph_\omega=\lambda$, while having the tree property at all $\aleph_n$, $n>1$.

Let $A$ denote the set $\set{\kappa_i}{i<\omega}$, and let $e: A \to A$ be a function which satisfies for all $\alpha,\beta$ in $A$:
\bce[(i)]
\item $i<j<\omega \to e(\kappa_i)\le e(\kappa_j)$.
\item $e(\kappa_i)\ge\kappa_{i+2}$ for all $i<\omega$.
\ece
We say that $e$ is an Easton function on $A$ which respects the $\kappa_i$'s (condition (ii)).

\begin{theorem}\label{Th:main}
Assume $\GCH$ and let $\seq{\kappa_i}{i<\omega}$, $\lambda$, and $A$ be as above. Let $e$ be an Easton function on $A$ which respects the $\kappa_i$'s. Then there is a forcing notion $\bb{Z}$ such that if $G$ is a $\Z$-generic filter, then in $V[G]$:
\bce[(i)]
\item Cardinals in $A$ are preserved, and all other cardinals below $\lambda$ are collapsed; in particular, for all $n<\omega$, $\kappa_n = \aleph_{n}$,
\item The continuum function on $A = \set{\aleph_n}{n<\omega}$ is controlled by $e$,i.e.\ $\forall n<\omega, 2^{\aleph_n}=e(\aleph_n)$.
\item The tree property holds at every $\aleph_n$, $2 \le n < \omega$.
\ece
\end{theorem}

For obtaining the model we are using the Cummings-Foreman iteration from \cite{CUMFOR:tp} followed by the Easton product of Cohen forcings which live in suitable inner models.

\subsection{The forcing}
Let $e$ be an Easton function on $A$ which respect the $\kappa_n$'s and let $\R_\omega$ be the forcing from Cummings and Foreman. Our forcing $\bb{Z}$ is defined as follows:

\beq
\bb{Z} = \R_\omega * \prod_{n<\omega} \Add(\kappa_n,e(\kappa_n))^{V[\R_{n-1}]},
\eeq
where we identify $V[\R_{-1}]$ (for $n = 0$) with $V$.

Let us denote this product by $\E$ and let $\dot{\E}$ be a canonical $\R_\omega$-name for it. We can therefore write \beq \bb{Z} = \R_\omega*\dot{\E}.\eeq Now we need to verify that the tree property holds in this model below $\aleph_\omega$ and that the continuum function is represented by $e$.

\subsection{The right continuum function}\label{sub:contf}

In this section, we show that $\bb{Z}$ forces the right continuum function:

\begin{theorem}\label{th:contf}
$\R_\omega*\dot{\E}$ forces that for all $n<\omega$, $\kappa_n=\aleph_n$ and $2^{\kappa_n}=e(\kappa_n)$.
\end{theorem}

We prove the theorem in a series of lemmas. Before we begin with the analysis of the forcing $\R_{\omega}*\dot{\E}$, let us fix some notation. For $n<\omega$ let $\dot{\R}_{[n,\omega)}$ denote the canonical $\R_n$-name for the tail $\R_{[n,\omega)}$ of the iteration $\R_\omega$. If $i<n$ let as also denote $\dot{\R}_{[i,n)}$ the canonical $\R_i$-name for the iteration between $i$ and $n$, $\R_{[i,n)}$.

In $V[\R_\omega]$, let us denote by $\P_n^\E$ the Cohen forcing $\Add(\kappa_n,e(\kappa_n))^{V[\R_{n-1}]}$ in the product $\E$, $n<\omega$. Moreover, let us denote by $\E_n$ the product of first $n$-many Cohen forcings in $\E$, i.e.\ $\E_n=\prod_{i<n} \P^\E_i$ and analogously let $\E_{[n,\omega)}$ denote the product of the rest of the forcing, i.e.\ $\E_{[n,\omega)}=\prod_{i\ge n}\Add(\kappa_i,e(\kappa_i))^{V[\R_{i-1}]}$; we have $\E \iso \E_n\times \E_{[n,\omega)}$. Let us further define $\E_{(j,n)}=\prod_{j<i<n} \Add(\kappa_i,e(\kappa_i))^{V[\R_{i-1}]}$ for $j\le n$ and let $\dot{\E}_n$, $\dot{\E}_{[n,\omega)}$ and $\dot{\E}_{(j,n)}$ denote the canonical $\R_\omega$-name for $\E_n$, $\E_{[n,\omega)}$ and $\E_{(j,n)}$-name, respectively.

It is easy to see that for all $n<\omega$, $\dot{\E}_{n+2}$ can be identified with an $\R_n$-name as all Cohen forcings in $\dot{\E}_{n+2}$ live in $V[\R_n]$. Therefore we can factor the iteration as $\R_\omega*\dot{\E}=\R_n*(\dot{\E}_{n+2}\times \dot{\R}_{[n,\omega)})*\dot{\E}_{[n+2,\omega)}$ for each $n<\omega$.

\begin{lemma}\label{L:R_n}
Let $n>0$. Then in $V[\R_n*\dot{\E}_{n+1}]$, the following hold:
\bce[(i)]
\item $\aleph_i=\kappa_i$ for $i<n+2$;
\item $2^{\kappa_i}=e(\kappa_i)$ for $i<n+1$.
\ece
\end{lemma}

\begin{proof}
(i).\ Let $n>0$ be given. First recall Cummings-Foreman result that for all $1<i<n+2$, $\kappa_i=\aleph_i$ in $V[\R_n]$, $2^{\kappa_{i-2}}=\kappa_i$ and $\GCH$ holds everywhere else.

We will show by induction starting with $i = n$ and descending to $0$ that for each $0 \le i \le n$, the forcing $\E_{[i,n+1)}$ behaves well over the model $V[\R_n]$ in the sense that it does not unintentionally collapse cardinals and forces the right continuum function. The assumptions for the induction are as follows:

\medskip

\bce[(a)]
\item $\E_{[i,n+1)} = \P^\E_i \x \E_{(i,n+1)}$  is $\kappa_{i-1}$-closed in $V[\R_n]$,
\item $\P^\E_i$ is $\kappa_i$-distributive in $V[\R_n][\E_{(i,n+1)}]$,
\item $\P^\E_i$  is $\kappa_{i+1}$-cc in $V[\R_n][\E_{(i,n+1)}]$.
\ece

\medskip

Notice that if we verify (a)--(c) for each $0 \le i \le n$, then the result follows because by stage $i =0$ we have dealt with the whole forcing $\E_{[0,n+1)} = \E_{n+1}$ (items (b) and (c) imply that for each $i$, $\P^\E_i$ preserves cardinals over the model $V[\R_n][\E_{(i,n+1)}]$, with (a) being a useful assumptions which keeps the induction running).

The base case is $i =n$, which means that $\P_n^\E$ should satisfy points (a)--(c) in $V[\R_n]$. This is true by Lemma \ref{L:closed-dis}(ii), Lemma \ref{TPd}(i)(with a trivial forcing $R$) and Lemma \ref{L:K-cc}(i), respectively.

For the induction step, let us assume that (a)--(c) hold for $0 < i+1 \le n$, and we will verify (a)--(c) for $i$.

\begin{itemize}
\item[(a)] It suffices to show that $\P_i^{\E}$ is $\kappa_{i-1}$-closed in $V[\R_n][\E_{(i,n+1)}]$ because by the induction assumption (a), $\E_{(i,n+1)} = \E_{[i+1,n+1)}$ is $\kappa_i$-closed in $V[\R_n]$.
~

The forcing $\R_n$ is equal to $\R_{i-1}*\dot{\R}_{[i-1,n)}$ and $\dot{\R}_{[i-1,n)}$ is forced to be $\kappa_{i-1}$-distributive by Fact \ref{F:iterace}(ii). Therefore $\P_i^{\E}$ is $\kappa_{i-1}$-closed in $V[\R_n]$ by Lemma \ref{L:closed-dis}(ii). 

\item[(b)] We wish to show that $\P_i^{\E}$ is $\kappa_{i}$-distributive in $V[\R_n][\E_{(i,n+1)}]$. 
~

$\R_n$ can be written as 
\beq
\R_{i-1}*\dot{\Q}_{i-1}*\dot{\Q}_{i}*\dot{\R}_{[i+1,n)}.
\eeq
Working in $V[\R_{i-1}]$, $\Q_{i-1}*\dot{\Q}_{i}$ is short for $\R(\kappa_{i-1},\kappa_{i+1},V[\R_{i-2}],V[\R_{i-1}],F^*_{i+1})*\R(\kappa_{i},\kappa_{i+2},V[\R_{i-1}],V[\R_{i}],F^*_{i+2})$ and this forcing is forcing equivalent to 
\beq
(\R(\kappa_{i-1},\kappa_{i+1},V[\R_{i-2}],V[\R_{i-1}],F^*_{i+1})\times\P_i)*\dot{\D}_i
\eeq
where $\dot{\D}_i$ is forced to be $\kappa_i$-closed after $\R(\kappa_{i-1},\kappa_{i+1},V[\R_{i-2}],V[\R_{i-1}],F^*_{i+1})\times\P_i)$ by Fact \ref{F:iterace}(vi).  But $\P_i^\E$ is $\kappa_i$-distributive after the forcing $$\R(\kappa_{i-1},\kappa_{i+1},V[\R_{i-2}],V[\R_{i-1}],F^*_{i+1})\times\P_i)$$ by Lemma \ref{TPd}(iv), therefore we can apply Lemma \ref{L:closed-dis}(i) to $\dot{\D}_i$ and $\P_i^\E$ and conclude that $\P_i^\E$ is $\kappa_i$-distributive in $V[\R_{i-1}][\R(\kappa_{i-1},\kappa_{i+1},V[\R_{i-2}],V[\R_{i-1}],F^*_{i+1})\times\P_i)*\dot{\D}_i]$. The rest of the proof again follows by  Lemma \ref{L:closed-dis}(i) from Fact \ref{F:iterace}(ii) that $\R_{[i+1,n)}$ is $\kappa_i$-closed and from the induction hypothesis that $\E_{(i,n+1)}$ is $\kappa_i$-closed in $V[\R_n]$.

\item[(c)] We wish to show that $\P_i^\E$ is $\kappa_{i+1}$-cc in $V[\R_n][\E_{(i,n+1)}]$. 
~

The forcing $\R_n*\dot{\E}_{(i,n+1)}$ is forcing equivalent to 
\beq
\R_{i-1}*\dot{\Q}_{i-1}*(\dot{\Q}_i\times \dot{\P}_{i+1}^\E)*\dot{\Q}_{i+1}*\dot{\R}_{(i+1,n)} *\dot{\E}_{(i+1,n+1)}.
\eeq

As both $\P_i^{\E}$ and $\Q_{i-1}$ are $\kappa_{i+1}$-Knaster in $V[\R_{i-1}]$, $\Q_{i-1}$ forces that $\P_{i}^\E$ is $\kappa_{i+1}$-cc and thus $\P_i^{\E}$ is $\kappa_{i+1}$-cc in $V[\R_i]$. Now, in $V[\R_i]$, $(\Q_i\times \P_{i+1}^\E)*\dot{\Q}_{i+1}$ is forcing equivalent to 
\beq\label{eq:forcing_cc}
(\Q_i\times \P_{i+1}^\E\times \P_{i+1})*\dot{\D}_{i+1},
\eeq
where $\dot{\D}_{i+1}$ is a $\Q_i\times\P_{i+1}$-name for a forcing notion which is $\kappa_{i+1}$-closed. As $\P_{i+1}^\E$ stays $\kappa_{i+1}$-distributive after $\Q_i\times\P_{i+1}$ by Lemma \ref{TPd}(iv), $\dot{\D}_{i+1}$ is still forced to be $\kappa_{i+1}$-closed after forcing with $\P_{i+1}^\E$ by Lemma \ref{L:closed-dis}(ii).

Our forcing $\P_i^\E$ is still $\kappa_{i+1}$-cc after $\Q_i\times \P_{i+1}^\E\times \P_{i+1}$ by Lemma \ref{TPd}(ii). By the previous paragraph and Lemma \ref{L:Easton}(ii) it is still $\kappa_{i+1}$-cc after the forcing (\ref{eq:forcing_cc}), which is forcing equivalent to $(\Q_i*\dot{\Q}_{i+1})\times \P_{i+1}^\E$. 

In $V[\R_{i+2}]$, $\P_{i+1}^\E$ is $\kappa_{i+1}$-distributive and $\R_{(i+1,n)}$ is $\kappa_{i+1}$-closed by Fact \ref{F:iterace}(ii) and thus $\R_{(i+1,n)}$ is still $\kappa_{i+1}$-closed in $V[\R_{i+2}][\P_{i+1}^\E]$ by Lemma \ref{L:closed-dis}(ii). Therefore our forcing $\P_i^\E$ is $\kappa_{i+1}$-cc in 
\beq\label{eq:model}
V[\R_{i+2}][\P_{i+1}^\E][\R_{(i+1,n)}]=V[\R_n][\P_{i+1}^\E]
\eeq
by Lemma \ref{L:Easton}(ii)

Now it is enough to realize that by the  induction hypothesis $\E_{(i+1,n+1)}$ is $\kappa_{i+1}$-closed in $V[\R_n]$ and $\P_{i+1}^\E$ is $\kappa_{i+1}$-distributive and thus $\E_{(i+1,n+1)}$ is $\kappa_{i+1}$-closed in the model (\ref{eq:model}) by Lemma \ref{L:closed-dis}(ii). Therefore we can apply Lemma \ref{L:Easton}(ii) to $\P_i^{\E}$ and $\E_{(i+1,n+1)}$ over the model  (\ref{eq:model}), hence $\P_i^{\E}$ is $\kappa_{i+1}$-cc in $V[\R_n][\E_{(i,n+1)}]$.
\end{itemize}

(ii).\ Easily follows from (i).
\end{proof}

\begin{Cor}\label{Cor:cc}
Let $n<\omega$ be given. In $V[\R_n]$ the following hold:
\bce[(i)]
\item For $i<n$, $\E_{(i,n+1)}$ forces $\E_{i+1}$ is $\kappa_{i+1}$-cc.
\item For $i<n+1$, $\E_{i+1}$ is $\kappa_{i+1}$-cc, in particular $\E_{n+1}$ is $\kappa_{n+1}$-cc.
\ece
\end{Cor}

\begin{proof}
This is immediate from proof of (c) of the previous lemma using Lemma \ref{L:iterace_property} and fact that chain condition is upward closed.
\end{proof}

\begin{lemma}
In $V[\R_\omega]$, $\E_{[n,\omega)}$ is $\kappa_{n-1}$-closed for each $n>0$.
\end{lemma}

\begin{proof}
Let $n>0$ be given. As the product of $\kappa_{n-1}$-closed forcings is $\kappa_{n-1}$-closed, it suffices to show that for each $i\ge n$, $\P^\E_i=\Add(\kappa_i,e(\kappa_i))^{V[\R_{i-1}]}$ is $\kappa_{n-1}$-closed. 
$\P^\E_i$ is defined in $V[\R_{i-1}]$ and it is even $\kappa_i$-closed there, but $\R_{[i-1,\omega)}$, the tail of the iteration $\R_\omega$, is just $\kappa_{i-1}$-distributive in $V[\R_{i-1}]$\footnote{To see that $\R_{[i-1,\omega)}$ is $\kappa_{i-1}$ -distributive, note that $\R_{[i-1,\omega)}=\Q_{i-1}*\R_{[i,\omega)}$ and $\Q_{i-1}$ is $\kappa_{i-1}$-distributive and forces that $\R_{[i,\omega)}$ is $\kappa_{i-1}$-closed by Fact \ref{F:iterace}(ii).}, and therefore $\P_i$ remains $\kappa_{i-1}$-closed in $V[\R_\omega]$ and thus at least $\kappa_{n-1}$-closed.
\end{proof}

\begin{lemma}\label{L:sequence}
For each $0 \le n < \omega$, any $\kappa_n$-sequence of ordinals in $V[\R_\omega][\E]$ is already added by $\R_n*(\dot{\P}_n\times\dot{\E}_{n+1})$.
\end{lemma}

\begin{proof}
Let $n\ge 0$ be given. First note that by Fact \ref{F:iterace}(ii), $\R_{[n+2,\omega)}$ is $\kappa_{n+1}$-closed in $V[\R_{n+2}]$ and $\E_{[n+2,\omega)}$ is $\kappa_{n+1}$-closed in $V[\R_\omega]$, therefore $\R_{[n+2,\omega)}*\dot{\E}_{[n+2,\omega)}$ is $\kappa_{n+1}$-closed in $V[\R_{n+2}]$ and thus also in $V[\R_{n+2}][\P^\E_{n+1}]$ by Lemma \ref{L:closed-dis}(ii) as $\P^\E_{n+1}$ is $\kappa_{n+1}$-distributive in $V[\R_{n+2}]$. By Corollary \ref{Cor:cc}(i), $\E_{n+1}$ is $\kappa_{n+1}$-cc in $V[\R_{n+2}][\P^\E_{n+1}]$, therefore by Lemma \ref{L:Easton}(i), $\R_{[n+2,\omega)}*\dot{\E}_{[n+2,\omega)}$ is $\kappa_{n+1}$-distributive in $V[\R_{n+2}][\P^\E_{n+1}][\E_{n+1}]=V[\R_{n+2}][\E_{n+2}]$. Hence any $\kappa_n$-sequence of ordinals is already added by $\R_{n+2}*\dot{\E}_{n+2}$.

Now, work in $V[\R_n]$. The forcing $\Q_n*\dot{\Q}_{n+1}$ is forcing equivalent to $(\Q_n\times \P_{n+1})*\dot{\D}_{n+1}$, where $\dot{\D}_{n+1}$ is forced to be $\kappa_{n+1}$-closed and stays $\kappa_{n+1}$-closed after forcing with $\P_{n+1}^\E$ by Lemma \ref{TPd}(iv) and Lemma \ref{L:closed-dis}(ii). Now we can apply Lemma \ref{L:Easton}(i) over $V[\R_n][\Q_n\times \P_{n+1}\times \P_{n+1}^\E]$ to $\E_{n+1}$\footnote{Note that $\E_{n+1}$ is $\kappa_{n+1}$-cc in $V[\R_n]$ by Corollary \ref{Cor:cc} and it remains $\kappa_{n+1}$-cc over the present model by Lemma \ref{TPd}(ii).} and $\D_{n+1}$ to show that $\D_{n+1}$ is $\kappa_{n+1}$- distributive in $V[\R_n][\Q_n\times \P_{n+1}\times \P_{n+1}^\E][\E_{n+1}]=V[\R_{n+1}][\P_{n+1}][\E_{n+2}]$. Therefore any $\kappa_n$-sequence of ordinals is already added by $\R_{n+1}*\dot{\P}_{n+1}*\dot{\E}_{n+2}$.

Work again in $V[\R_n]$. $\E_{n+1}$ is $\kappa_{n+1}$-cc and $\P_{n+1}\times \P_{n+1}^\E$ is $\kappa_{n+1}$-closed here, therefore by Lemma \ref{TPd}(i) $\P_{n+1}\times \P_{n+1}^\E$ is $\kappa_{n+1}$-distributive in $V[\R_n][\Q_n][\E_{n+1}]=V[\R_{n+1}][\E_{n+1}]$. Therefore any $\kappa_n$-sequence of ordinals is already in $V[\R_{n+1}][\E_{n+1}]$.

In $V[\R_n]$, $\Q_n$ is a projection of $\P_n\times \T_n$, where $\T_n$ is $\kappa_{n+1}$-closed and $\P_n$ is $\kappa_{n+1}$-Knaster, therefore $\E_{n+1}\times \P_n$ is $\kappa_{n+1}$-cc and hence $\T_n$ stays $\kappa_{n+1}$-distributive after forcing with $\E_{n+1}\times \P_n$ by Lemma \ref{L:Easton}(i). It follows that every $\kappa_n$-sequence is added by $\R_n*( \dot{\P}_n\times\dot{\E}_{n+1})$, as desired.
\end{proof}

Now we can finish the proof of Theorem \ref{th:contf}:

\begin{proof}
(Proof of theorem \ref{th:contf}.) The theorem follows from Lemma \ref{L:sequence}, Lemma \ref{L:R_n} and the fact that $\P_n\times \E_{n+1}$ is isomorphic to $\E_{n+1}$ over $V[\R_n]$.
\end{proof}

\subsection{The tree property}\label{sub:tp}

In this section we finish the argument by showing:

\begin{theorem}\label{th:tp}
$\R_\omega*\dot{\E}$ forces that the tree property holds at $\kappa_{n+2}$, for every $n\ge 0$.
\end{theorem}

We prove the theorem in two subsections and several lemmas. Let us fix some $n\ge 0$, and let us denote $\kappa_{n+2}$ by $\kappa$. We show the tree property at $\kappa$. 

In $V[\R_{n+2}]$, let $\E_{n+3}|\kappa$ be the product $\prod_{i<n+3} \Add(\kappa_i,\lambda_i)^{V[\R_{i-1}]}$, where $\lambda_i=\kappa$ for $e(\kappa_i)>\kappa$ and $\lambda_i=e(\kappa_i)$ otherwise.

\begin{lemma}\label{C:small}
If $\R_\omega*\dot{\E}$ adds a $\kappa$-Aronszajn tree, so does $\R_{n+2}*\dot{\E}_{n+3}|\kappa$.
\end{lemma}

\begin{proof}
Assume for contradiction that there is a $\kappa$-Aronszajn tree $T$ in generic extension by $\R_\omega*\dot{\E}$. By Lemma \ref{L:sequence}, $T$ has to be added by $\R_{n+2}*(\dot{\P}_{n+2}\times\dot{\E}_{n+3})$ and as this forcing is isomorphic to $\R_{n+2}*\dot{\E}_{n+3}$, $T$ is in the generic extension by $\R_{n+2}*\dot{\E}_{n+3}$. 

Now, work in $V[\R_{n+2}]$. In this model $\kappa^+=\kappa_{n+3}=\aleph_{n+3}$ and by Lemma \ref{L:R_n}, $\E_{n+3}$ is $\kappa^+$-cc. Therefore there is a nice $\E_{n+3}$-name $\dot{T}$ for $T$ of size $\kappa$. Such a nice name contains at most $\kappa$-many conditions in $\E_{n+3}$, hence we can restrict each $\Add(\kappa_i,e(\kappa_i))$ (if necessary) in the product $\E_{n+3}$ to $\Add(\kappa_i,A_i)$, where $A_i$ has size at most $\kappa$ and it is determined by the support of conditions in $\dot{T}$. The claim now follows as any bijection between $A_i$ and $\kappa$ gives an isomorphism between $\Add(\kappa_i,A_i)$ and $\Add(\kappa_i,\kappa)$.
\end{proof}

Let us denote $\R_{n+2}*\dot{\E}_{n+3}|\kappa$ by $\R_{n+2}*\dot{\E}_{n+3}$  in the interest of brevity and let us keep in mind that all the Cohen forcings in $\E_{n+3}$ have length less than or equal to $\kappa$. 

Let us fix some notation now. Let $G_i$ denote a $\Q_i$-generic over $V[G_0][\dots][G_{i-1}]$, for each $i<n+2$, and $x_i$ a $\P^\E_i$-generic over $V[G_0][\dots][G_{n+1}][x_0][\dots][x_{i-1}]$ for each $i<n+3$. Let us denote by $V_{n-1}$ the model$V[G_0][\dots][G_{n-1}]$ and let us write for brevity $x_{<i}$ instead of $x_0\x\dots\x x_{i-1}$ for $i\le n+3$.

\subsubsection{Lifting an embedding}

We wish to lift an appropriate embedding to the model $V_{n-1}[G_n][G_{n+1}][x_{<n+3}]$ which contains the tree $T$.

In $V$, using the Laver function $F_{n+2}$, let us choose a supercompact embedding $j:V\then M$ such that:

\medskip

\bce[(i)]
\item $\mbox{crit}(j)=\kappa$, $j(\kappa)>\lambda$ and ${^\lambda M}\sub M$.\footnote{Recall that $\lambda$ is the limit of the sequence of the supercompact cardinals $\seq{\kappa_n}{n<\omega}$.}
\item $j(F_{n+2})(\kappa)$ is the canonical $\R_{n}$-name for the canonical $\Q_n$-name for $\T_{n+1}\times\P^\E_{n+2}$.
\ece

\medskip

We are going lift $j$ first to the model $V_{n-1}[G_n][G_{n+1}][x_{n+2}]$. The argument is essentially the same as in \cite{CUMFOR:tp}, except that we have the extra forcing $\P^\E_{n+2}$. Let us review the basic steps of the lifting. 

\begin{itemize}
\item As $j(\R_n)=\R_n$, we can lift the embedding from $V_{n-1}$ to $M_{n-1}=M[G_0][\dots][G_{n-1}]$.
\item Since $j$ is identity below $\kappa=\kappa_{n+2}$, $j(\Q_n)|\kappa=\Q_n$ and we can lift the embedding further from $V_{n-1}[G_n]$ to $M_{n-1}[G_n][h_n]$ in $V_{n-1}[G_n][h_n]$, where $h_n$ is $j(\Q_n)/G_n$-generic over $V_{n-1}[G_n]$.
\item Now work in $V_{n-1}[G_n][h_n]$ and define:
\beq
G^1_{n+1}\times x_{n+2}=\set{f(\kappa)^{G_n}}{\mbox{ for some }p,q, (p,q,f)\in G_n*h_n}.
\eeq
By our choice of $j$, $G^1_{n+1}\times x_{n+2}$ is $\T_{n+1}\times\P^\E_{n+2}$-generic over $V_{n-1}[G_n]$. 

By the projection $\sigma^{j(\kappa)}_\kappa$ (see the analysis below Remark \ref{T_Collaps}), $V_{n-1}[G_n][h_n]=V_{n-1}[G_n][G^1_{n+1}\times x_{n+2}][h^*_n]$ for some $j(\Q_n)/(G_n*(G^1_{n+1}\times x_{n+2}))$-generic filter $h^*_n$.

The family of condition $j''(G^1_{n+1}\times x_{n+2})$ has a lower bound $t= ((\emptyset, p_m,q_m),t_m)$ in the product forcing $j(\T_{n+1}) \x j(\P^\E_{n+2})$ because $j(\T_{n+1}\times\P^\E_{n+2})$ is $j(\kappa)$-directed closed and $j(\kappa)>\lambda>\kappa_{n+3}$. The condition $t$ can be used as a master condition for $j$ and $\Q_{n+1}\times\P_{n+2}^\E$: if $H_{n+1}\times y_{n+2}$ is $j(\Q_{n+1})\times j(\P^{\E}_{n+2})$-generic over $V_{n-1}[G_n][h_n]$ and $H_{n+1}$ contains $(\emptyset,p_m,q_m)$ and $y_{n+2}$ contains $t_m$, then ${j^{-1}}'' (H_{n+1} \x y_{n+2})$ generates a $\Q_{n+1}\times \P_{n+2}^\E$-generic over $V_{n-1}[G_n]$. Let us denote by $G_{n+1}\times x_{n+2}$ the $\Q_{n+1}\times \P_{n+2}^\E$-generic over $V_{n-1}[G_n]$ generated by ${j^{-1}}'' (H_{n+1}\x y_{n+2})$.
\beq
G_{n+1}\times x_{n+2}=\pi_{n+1}'' (\rho_{n+1}''G_{n+1}\times G^1_{n+1})\times x_{n+2}.
\eeq
\end{itemize}

Therefore we can lift the embedding to \beq j: V_{n-1}[G_n][G_{n+1}][x_{n+2}] \then M_{n-1}[G_n][h_n][H_{n+1}][y_{n+2}]. \eeq Note that the model $M_{n-1}[G_n][h_n][H_{n+1}][y_{n+2}]$ is the same as $M_{n-1}[G_n][G^1_{n+1}\times x_{n+2}][h^*_n][H_{n+1}][y_{n+2}]$.

Now we need to lift $j$ further to $\E_{n+2}$. Since $j$ is identity below $\kappa$ and $\E_{n+2}=\prod_{i<n+2}\P_i^\E$, $j$ is the identity on conditions in $\E_{n+2}$. For each $i<n+2$, $j(\P^\E_i)=\P^\E_i\times j(\P^\E_i)|[\kappa,j(\kappa))$\footnote{Note that $ j(\P^\E_i)|[\kappa,j(\kappa))$ is isomorphic to $j(\P^\E_i)$ therefore for simplification of the notation we will write $j(\P^\E_i)$ instead of $ j(\P^\E_i)|[\kappa,j(\kappa))$}. Therefore we can lift the embedding further from the model  $V_{n-1}[G_n][G_{n+1}][x_{n+2}][x_{<n+2}]$ to $M_{n-1}[G_n][h_n][H_{n+1}][y_{n+2}][y_{<n+2}]$, where 
\begin{itemize}
\item $y_{<n+2}$ denotes $y_0\times \dots\times y_{n+1}$ and
\item for each $i<n+2$ there is $x^*_i$ such that $y_i=x_i\times x^*_i$ and $y_i$ is $j(\P^\E_i)$-generic over $V_{n-1}[G_n][h_n][H_{n+1}][y_{n+2}][y_{<i}]$.
\end{itemize}

Let us write the model $M_{n-1}[G_n][h_n][H_{n+1}][y_{n+2}][y_{<n+2}]$ equivalently as 
\beq\label{M:2}
M_{n-1}[G_n][G^1_{n+1}\times x_{n+2}][h^*_n][H_{n+1}][y_{n+2}][y_{<n+2}].
\eeq

We will rearrange the generics to be able to argue for the tree property in the next section. 

$H_{n+1}$ is $j(\Q_{n+1})$-generic over the $M_{n-1}[G_n][G^1_{n+1}\times x_{n+2}][h^*_n]$ and by applying the projection $\rho^*_{n+1}:j(\Q_{n+1})\then j(\P_{n+1})$ we get a $j(\P_{n+1})$-generic; let us denote it by $H^0_{n+1}$ and let us also denote by $H^1_{n+1}$ a $j(\D_{n+1})=j(\Q_{n+1})/H^0_{n+1}$-generic over  $M_{n-1}[G_n][G^1_{n+1}\times x_{n+2}][h^*_n][H^0_{n+1}]$ such that $H_{n+1}=H^0_{n+1}*H^1_{n+1}$. Now the model (\ref{M:2}) is equal to 
\beq\label{M:3}
M_{n-1}[G_n][G^1_{n+1}\times x_{n+2}][h^*_n][H^0_{n+1}][H^1_{n+1}][y_{n+2}][y_{<n+2}].
\eeq

The elementary embedding $j$ is in particular a regular embedding from $\P_{n+1}$ to $j(\P_{n+1})$ and therefore ${j^{-1}}''H^0_{n+1}$ yields a generic filter for  $\P_{n+1}$ over $M_{n-1}[G_n][G^1_{n+1}\times x_{n+2}][h^*_n]$. Let us denote this generic by $G^0_{n+1}$ and let $h^0_{n+1}$ be a generic filter such that $G^0_{n+1}\x h^0_{n+1}=H^0_{n+1}$. Therefore the model (\ref{M:3}) can be decomposed further as 
\beq\label{M:4}
M_{n-1}[G_n][G^1_{n+1}\times x_{n+2}][h^*_n][G^0_{n+1}][h^0_{n+1}][H^1_{n+1}][y_{n+2}][y_{<n+2}].
\eeq

Now note that $\P_{n+1}$ lives already in $M_{n-1}$ and as $G^0_{n+1}$ is generic over the model $M_{n-1}[G_n][G^1_{n+1}\times x_{n+2}][h^*_n]$, $G^0_{n+1}$ and $h^*_n$ are mutually generic over $M_{n-1}[G_n][G^1_{n+1}\times x_{n+2}]$ and also $G^0_{n+1}$, $G^1_{n+1}$ and $x_{n+2}$ are mutually generic over $M_{n-1}[G_n]$. Therefore we can rearrange model (\ref{M:4}) as
\beq\label{M:5}
M_{n-1}[G_n][G^0_{n+1}\x G^1_{n+1}\times x_{n+2}][h^*_n][h^0_{n+1}][H^1_{n+1}][y_{n+2}][y_{<n+2}].
\eeq
Recall that there is the projection $\pi_{n+1}: \P_{n+1}\times \T_{n+1}\then \Q_{n+1}$.\footnote{$\pi_{n+1}'' (G^0_{n+1}\x G^1_{n+1})=G_{n+1}$} Therefore we can rewrite the model (\ref{M:5}) as 
\beq\label{M*}
M_{n-1}[G_n][G_{n+1}][G_{\S}][x_{n+2}][h^*_n][h^0_{n+1}][H^1_{n+1}][y_{n+2}][y_{<n+2}],
\eeq
where $G_\S$ is $\S_{n+1}$-generic over $M_{n-1}[G_n][G_{n+1}]$ such that $G^0_{n+1}\x G^1_{n+1}=G_{n+1}*G_\S$. Recall that $\S_{n+1}$ is the quotient forcing $\P_{n+1}\times \T_{n+1}/G_{n+1}$. 

Finally, for each $i<n+2$, $y_i=x_i\x x^*_i$, hence we can write the model (\ref{M*}) as follows:
\beq\label{M:6}
M_{n-1}[G_n][G_{n+1}][G_{\S}][x_{n+2}][h^*_n][h^0_{n+1}][H^1_{n+1}][y_{n+2}][x_0\times x^*_0][\dots][x_{n+1}\x x^*_{n+1}],
\eeq
and again by mutual genericity we can rearrange the generic filters in (\ref{M:6}) as follows:
\beq\label{M_1}
M_{n-1}[G_n][G_{n+1}][x_{<n+3}][x^*_{<n+2}][h^0_{n+1}][G_{\S}][h^*_n][H^1_{n+1}][y_{n+2}].
\eeq

\subsubsection{The tree property argument}

Recall that we assume that  $T$ is $\kappa$-Aronszajn tree in $V_{n-1}[G_n][G_{n+1}][x_{<n+3}]$. By the closure properties of the models, we can assume that $T$ is also in $M_{n-1}[G_n][G_{n+1}][x_{<n+3}]$. As $j(T)\rest\kappa=T$, $T$ has a cofinal branch in model (\ref{M_1}). We will argue that the forcing from $M_{n-1}[G_n][G_{n+1}][x_{<n+3}]$ to the model (\ref{M_1}) cannot add a cofinal branch to $T$ over $M_{n-1}[G_n][G_{n+1}][x_{<n+3}]$. This will contradict the assumption that $T$ ia a $\kappa$-Aronszajn tree in $V_{n-1}[G_n][G_{n+1}][x_{<n+3}]$, and conclude the whole proof.

First we show that there are no cofinal branches in $T$ in the smaller model:
\beq\label{M_3}
M_{n-1}[G_n][G_{n+1}][x_{<n+3}][x^*_{<n+2}][h^0_{n+1}][G_{\S}][h^*_n].
\eeq
Let us work for a while in $M_{n-1}[G_n][G^1_{n+1}\times x_{n+2}]$;  $h^*_n$ is $j(\Q_n)/(G_n*(G^1_{n+1}\times x_{n+2}))$-generic over this model and there is a projection $\pi^*_n:j(\P_n)\times\T^*_n\then j(\Q_n)/(G_n*(G^1_{n+1}\times x_{n+2}))$. Therefore we can find $h^{*0}_n\times h^{*1}_n$ which is $j(\P_n)\times\T^*_n$-generic over $$M_{n-1}[G_n][G_{n+1}][x_{<n+3}][x^*_{<n+2}][G_{\S}][h^0_{n+1}]$$ such that ${\pi_n^*}''(h^{*0}_n\times h^{*1}_n)=h^*_n$.

In order to argue that there are no cofinal branches through $T$ in the model (\ref{M_3}), it is enough to show that there are no such branches in the larger model:
\beq\label{M_2}
M_{n-1}[G_n][G_{n+1}][x_{<n+3}][x^*_{<n+2}][h^{*0}_n][h^0_{n+1}][G_{\S}][h^{*1}_n].\footnote{Note that in contrast to $h^{*1}_n$, we can put $h^{*0}_n$ before $G_\S$ as it is generic for the Cohen forcing $j(\P_n)$ and it already lives in $V_{n-2}$.}
\eeq

We divide the proof of the proposition that $T$ has no cofinal branch in (\ref{M_2}) into two claims: First we use the $\kappa$-square-cc of the Cohen forcings which add the generic $x^*_{<n+2}\x h^{*0}_n\x h^0_{n+1}$ to show that they do not add cofinal branches to $T$, and then we use the closure property of forcings which add $G_{\S}*h^{*1}_n$ to show that they cannot add a cofinal branch to $T$ either.

\begin{Claim}\label{Cl:Cohens}
$j(\E_{n+2})\x j(\P_n)\x j(\P_{n+1})$ is $\kappa$-square-cc in $M_{n-1}[G_n][G_{n+1}][x_{<n+3}]$.
\end{Claim}

\begin{proof}
First note that the product $j(\E_{n+2})\x j(\P_n)\x j(\P_{n+1})$ is isomorphic to $j(\E_{n+1})\x j(\P_{n+1})$ as $\P^\E_n\x \P_n$ is isomorphic to $\P^\E_n$, and $\P^\E_{n+1}\x\P_{n+1}$ is isomorphic to $\P_{n+1}$\footnote{Note that $\P_{n+1}$ has length $\kappa_{n+3}$ hence $\P^\E_{n+1}\x\P_{n+1}$ is not isomorphic to $\P^\E_{n+1}$ as this has length less or equal $\kappa$}. Also note that $j(\E_{n+1})\x j(\P_{n+1})$ is isomorphic to its square. Hence to show that $j(\E_{n+1})\x j(\P_{n+1})\x j(\E_{n+1})\x j(\P_{n+1})$ is $\kappa$-cc, it suffices to show that $j(\E_{n+1})\x j(\P_{n+1})$ is $\kappa$-cc.

In $M_{n-1}[G_n][G_{n+1}][x_{n+2}]$, $\E_{n+2}\x j(\E_{n+1})\x j(\P_{n+1})$ is isomorphic to $j(\E_{n+1})\x j(\P_{n+1})$; if we show that $j(\E_{n+1})\x j(\P_{n+1})$ is $\kappa$-cc in this model, we conclude that that $\E_{n+2}\x j(\E_{n+1})\x j(\P_{n+1})$ is $\kappa$-cc, i.e.\ $\E_{n+2}$ forces that $j(\E_{n+1})\x j(\P_{n+1})$ is $\kappa$-cc, which implies $j(\E_{n+1})\x j(\P_{n+1})$ is $\kappa$-cc in $M_{n-1}[G_n][G_{n+1}][x_{<n+3}]$.

To show that $j(\E_{n+1})\x j(\P_{n+1})$ is $\kappa$-cc in $M_{n-1}[G_n][G_{n+1}][x_{n+2}]$, we proceed as in the proof of Lemma \ref{L:R_n}(c).
\end{proof}

Since $j(\E_{n+2})\x j(\P_n)\x j(\P_{n+1})$ is $\kappa$-square-cc in $M_{n-1}[G_n][G_{n+1}][x_{<n+3}]$, there are no cofinal branches through $T$ in 
\beq
M_{n-1}[G_n][G_{n+1}][x_{<n+3}][x^*_{<n+2}][h^{*0}_n][h^0_{n+1}],
\eeq 
by Fact \ref{F:Square_cc}

\begin{Claim}\label{Cl:Closed}
In the model $M_{n-1}[G_n][G_{n+1}][x_{n+2}][y_{n+1}][h^0_{n+1}]$ the following hold:
\bce[(i)]
\item $\S_{n+1}*\T_n^*$ is $\kappa_{n+1}$-closed.
\item $\E_{n+1}\x j(\E_{n+1})\x j(\P_n)$ is $\kappa_{n+1}$-cc.
\ece
\end{Claim}

\begin{proof}
(i) The forcing $\S_{n+1}$ lives in $M_{n-1}[G_n][G_{n+1}]$ and it is $\kappa_{n+1}$-closed there, but it is also $\kappa_{n+1}$-closed in $M_{n-1}[G_n][G_{n+1}][x_{n+2}]$ by Lemma \ref{L:closed-dis}(ii) as $\P^\E_{n+2}$ is $\kappa_{n+1}$-closed in $M_{n-1}[G_n][G_{n+1}]$ by Lemma \ref{L:R_n}(a).

Now, the term forcing $\T^*_n$ lives in $M_{n-1}[G_n][G^1_{n+1}\x x_{n+2}]$ and it is $\kappa_{n+1}$-closed there. The model $M_{n-1}[G_n][G_{n+1}][x_{n+2}][G_\S]$ is equal to  $M_{n-1}[G_n][G^1_{n+1}\x x_{n+2}\x G^0_{n+1}]$. Therefore to show that $\T^*_n$ is $\kappa_{n+1}$-closed here it is enough to show that it stay closed after forcing with $\P_{n+1}$, but this holds by Lemma \ref{L:closed-dis}(ii) as $\P_{n+1}$ is $\kappa_{n+1}$-distributive in  $M_{n-1}[G_n][G^1_{n+1}\x x_{n+2}]$.\footnote{ $\P_{n+1}$ is $\kappa_{n+1}$-distributive in $M_{n-1}[G_n]$ by Fact \ref{F:iterace}(v) and it stay $\kappa_{n+1}$-distributive by Lemma \ref{L:closed-dis}(i) after forcing with $\T_{n+1}\x \P^\E_{n+2}$ as this forcing is $\kappa_{n+1}$-closed in  $M_{n-1}[G_n]$.}

By the previous two paragraphs, $\S_{n+1}*\T^*_n$ is $\kappa_{n+1}$-closed in $M_{n-1}[G_n][G_{n+1}][x_{n+2}]$. Now, the product of Cohen forcings which add the generic filter $y_{n+1}\x h^0_{n+1}$ is isomorphic to $j(\P_{n+1})$. This forcing $j(\P_{n+1})$ is $\kappa_{n+1}$-distributive in $M_{n-1}[G_n][G_{n+1}][x_{n+2}]$ by Lemma \ref{L:R_n}(b); therefore the forcing $\S_{n+1}*\T^*_n$ remains $\kappa_{n+1}$-closed in the model  $M_{n-1}[G_n][G_{n+1}][x_{n+2}][y_{n+1}][h^0_{n+1}]$ by Lemma \ref{L:closed-dis}(ii) as required.

(ii) As before, the product $\E_{n+1}\x j(\E_{n+1})\x j(\P_n)$ is isomorphic to $j(\E_{n+1})$ and the proof that this forcing is $\kappa_{n+1}$-cc is as in the proof of Lemma \ref{L:R_n}(c).
\end{proof}

Now we can apply Fact \ref{F:ccc_Closed} to $\E_{n+1}\x j(\E_{n+1})\x j(\P_n)$ as $P$ and $\S_{n+1}*\T_n^*$ as $Q$ over the model $M_{n-1}[G_n][G_{n+1}][x_{n+2}][y_{n+1}][h^0_{n+1}]$. Therefore there are no cofinal branches in $T$ in the model (\ref{M_2}) and hence neither in the model (\ref{M_3}).

To finish the proof of the tree property at $\kappa$ it is enough to show that $j(\D_{n+1})\x j(\P_{n+2})$ cannot add a cofinal branch to $T$ over the model (\ref{M_3}).

\begin{Claim}\label{Cl:zbytek}
In the model $M_{n-1}[G_n][G_{n+1}][x_{n+2}][y_{n+1}][h^0_{n+1}][G_{\S}][h^*_n]$ the following hold:
\bce[(i)]
\item $j(\D_{n+1})\x j(\P_{n+2})$ is $\kappa_{n+1}$-closed.
\item $\E_{n+1}\x j(\E_{n+1})$ is $\kappa_{n+1}$-cc.
\ece
\end{Claim}

\begin{proof}
(i) First, the forcing $j(\D_{n+1})$ lives in $M_{n-1}[G_n][G_{n+1}][x_{n+2}][h^0_{n+1}][G_{\S}][h^*_n]$ and it is $\kappa_{n+1}$-closed there. 

Second, $j(\P_{n+2})$ lives in $M_{n-1}[G_n][G^1_{n+1}\x x_{n+2}][h^*_n]$ and it is $\kappa_{n+1}$-closed there. To get from model $M_{n-1}[G_n][G^1_{n+1}\x x_{n+2}][h^*_n]$ to $M_{n-1}[G_n][G^0_{n+1}\x G^1_{n+1}\x x_{n+2}][h^0_{n+1}][h^*_n]=M_{n-1}[G_n][G_{n+1}][x_{n+2}][G_{\S}][h^*_n][h^0_{n+1}]$ it suffices to force with $j(\P_{n+1})$, which adds a generic filter for $G^0_{n+1}\x h^0_{n+1}$. This forcing lives in $M_{n-1}$ and it is $\kappa_{n+1}$-distributive in $M_{n-1}[G_n][h_n]=M_{n-1}[G_n][G^1_{n+1}\x x_{n+2}][h^*_n]$ by Fact \ref{F:iterace}(v) or by Lemma \ref{TPd}. Therefore $j(\P_{n+2})$ remains $\kappa_{n+1}$-closed in $M_{n-1}[G_n][G_{n+1}][x_{n+2}][G_{\S}][h^*_n][h^0_{n+1}]$ by Lemma \ref{L:closed-dis}(ii).

As both forcings are $\kappa_{n+1}$-closed in 
\beq\label{M_4}
M_{n-1}[G_n][G_{n+1}][x_{n+2}][h^0_{n+1}][G_{\S}][h^*_n],
\eeq
their product $j(\D_{n+1})\x j(\P_{n+2})$ is $\kappa_{n+1}$-closed as well. The diference between the model (\ref{M_4}) and the model
\beq\label{M_5}
M_{n-1}[G_n][G_{n+1}][x_{n+2}][y_{n+1}][h^0_{n+1}][G_{\S}][h^*_n]
\eeq
-- where we want to show that $j(\D_{n+1})\x j(\P_{n+2})$ is $\kappa_{n+1}$-closed -- is just the forcing $j(\P^\E_{n+1})$ which adds the generic filter $y_{n+1}$. Therefore to finish the proof of the claim it suffices to show that $j(\P_{n+1})$ is $\kappa_{n+1}$-distributive in model (\ref{M_4}). The model (\ref{M_4}) is actually equal to
\beq
M_{n-1}[G_n][h_n][G^0_{n+1}][h^0_{n+1}].
\eeq

By Lemma \ref{TPd}(iv), $j(\P_{n+1})\x j(\P^\E_{n+1})$ is $\kappa_{n+1}$-distributive in $M_{n-1}[G_n][h_n][G^0_{n+1}]$. Therefore $j(\P_{n+1})$ forces that $j(\P^\E_{n+1})$ is $\kappa_{n+1}$-distributive and so $j(\P^\E_{n+1})$ is $\kappa_{n+1}$-distributive in the model (\ref{M_4}). Now we can apply Lemma \ref{L:closed-dis}(ii) to $j(\P^\E_{n+1})$ and $j(\D_{n+1})\x j(\P_{n+2})$ over the model (\ref{M_4}) and conclude that $j(\D_{n+1})\x j(\P_{n+2})$ is $\kappa_{n+1}$-closed in (\ref{M_5}).

(ii) Recall that the model $M_{n-1}[G_n][G_{n+1}][x_{n+2}][y_{n+1}][h^0_{n+1}][G_{\S}][h^*_n]$ is equal to 
\beq
M_{n-1}[G_n][h_n][G^0_{n+1}][h^0_{n+1}][y_{n+1}].
\eeq

The proof that $\E_{n+1}\x j(\E_{n+1})$  -- which is isomorphic to $j(\E_{n+1})$ -- is $\kappa_{n+1}$-cc in this model proceeds exactly as in the proof of Lemma \ref{L:R_n}(c).
\end{proof}

By the previous claim, we can apply Fact \ref{F:ccc_Closed} to $\E_{n+1}\x j(\E_{n+1})$ as $P$ and $j(\D_{n+1})\x j(\P_{n+2})$ as $Q$ over the model $M_{n-1}[G_n][G_{n+1}][x_{n+2}][y_{n+1}][h^0_{n+1}][G_{\S}][h^*_n]$ and conclude that there are no cofinal branches in $T$ in the model (\ref{M_1}). This is a contradiction which finishes the proof of Theorem \ref{th:tp}.

\section{Open questions}

For the first question below, let us assume $e: \omega \to \omega$ satisfies $n < m \then e(n) \le e(m)$ and $e(n)>n+1$ for all $n,m < \omega$.

\begin{que}
\emph{Is it possible to have the tree property at every $\aleph_n$, $1<n<\omega$, with $2^{\aleph_n} = \aleph_{e(n)}$, $n<\omega$, and $2^{\aleph_\omega} = \aleph_{\omega+m}$ for a prescribed $1 < m < \omega$?  (Note that in our model we have $2^{\aleph_\omega} = \aleph_{\omega+1}$.)}
\end{que}

A partial answer to this question was given by Honzik and Friedman in \cite{FH:tree1}, who showed that that $2^{\aleph_\omega}=\aleph_{\omega+2}$ is consistent with the tree property at every even cardinal below $\aleph_\omega$. However, this method does not seem to be appropriate for manipulating the continuum function as they used an iteration of the Sacks forcing, instead of the Mitchell forcing which allows greater flexibility. Unger \cite{S:along} extended this result using the Cummings-Foreman method to show that $2^{\aleph_\omega}=\aleph_{\omega+2}$ is consistent with the tree property at every cardinal $\aleph_n$ below $\aleph_\omega$, for $n>1$, with $2^{\aleph_n} = \aleph_{n+2}$ for each $n< \omega$.

\begin{que}
\emph{In our final model, can we in addition have the tree property at $\aleph_{\omega+2}$?}
\end{que}

Note that this question is still open even with the trivial continuum function; i.e. with $2^{\aleph_n}=\aleph_{n+2}$ for $n <\omega$.

\begin{que}
\emph{Can we control generalized cardinal invariants together with the tree property? For instance, is it possible to combine the results of Cummings and Shelah in \cite{CumShe:gen} for $\mathfrak{d}_\kappa$ and $\mathfrak{b}_\kappa$ with the tree property at relevant cardinals? }
\end{que}

\bigskip

\no {\bf Acknowledgements.} The author was supported by FWF/GA{\v C}R grant \emph{Compactness principles and combinatorics} (19-29633L).

\end{document}